\documentclass[10pt, reqno]{amsart}

\usepackage{microtype}
\usepackage[utf8]{inputenc}
\usepackage{lipsum} 
\usepackage{amsmath,amssymb,amsthm,amsfonts}
\usepackage{tikz-cd}
\usepackage{cleveref}
\usepackage[colorinlistoftodos]{todonotes}

\usepackage[top=1.2in,bottom=1.2in,left=1.1in,right=1.1in]{geometry}

\definecolor{myred}{RGB}{217,0,0}

\theoremstyle{plain}
\newtheorem{theorem}{Theorem}
\newtheorem{lemma}[theorem]{Lemma}

\newtheorem{definition}[]{Definition}

\newcommand{\inv}{^{-1}}

\DeclareMathOperator{\Stab}{Stab}
\DeclareMathOperator{\Orb}{Orb}

\DeclareMathOperator{\Out}{Out}
\DeclareMathOperator{\GL}{GL}
\DeclareMathOperator{\Image}{Im}

\newcommand{\Z}{\mathbb{Z}}

\newcommand{\C}{\mathbb{C}}

\title{Large Totally Symmetric Sets}
\author{Noah Caplinger}
\date{June 2022}

\begin{document}
\maketitle

\begin{abstract}
    A totally symmetric set is a subset of a group such that every permutation of the subset can be realized by conjugation in the group. The (non-)existence of large totally symmetric sets obstruct homomorphisms, so bounds on the sizes of totally symmetric sets are of particular use. In this paper, we prove that if a group has a totally symmetric set of size $k$, it must have order at least $(k+1)!$. We also show that with three exceptions, $\{(1 \; i)\mid i = 2,\ldots,n\} \subset S_n$ is the only totally symmetric set making this bound sharp; it is thus the largest totally symmetric set relative to the size of the ambient group.

    
\end{abstract}

\section{Introduction}

Kordek---Margalit \cite{Kordek-Margalit} introduced the notion of a totally symmetric set in a group as a means to study homomorphisms. Briefly, a subset $X \subset G$ of a group is totally symmetric if any permutation of $X$ can be realized by conjugation in $G$---for instance, the set of transpositions $$X_n = \{(1 \; i) \mid i = 2,\ldots,n \} \subset S_n$$ is totally symmetric. Understanding the totally symmetric sets of groups $G,H$ immediately yields constraints on homomorphisms $G\to H$, and in some cases give complete classifications. Kordek---Margalit \cite{Kordek-Margalit} classified homomorphisms $\rho:B_n' \to B_n$ with essentially this strategy: they first classify totally symmetric sets in $B_n$, then use this classification to deduce the image of a well-chosen totally symmetric set in $B_n'$. This general strategy has been used by Chen---Mukherjea \cite{Chen-Mukherjea} to classify maps from braid groups to mapping class groups, and by Scherich---Verberne \cite{Nancy-Yvon}, Caplinger---Kordek \cite{Caplinger-Kordek} and Kordek et al. \cite{Finite_Quotients} to understand finite quotients of braid groups. Classifications of totally symmetric sets and bounds on their sizes are of particular interest in this scheme, as they obstruct homomorphisms. Our two main results are directly in this vein.

\begin{theorem}
\label{thm:main}
Let $G$ be a group, and $X \subset G$ a totally symmetric set of cardinality $k>3$. Then $|G| \geq (k+1)!$. If $|G| = (k+1)!$, then $G \cong S_{k+1}$. 
\end{theorem}

This result should be compared to \cite[Proposition 2.2]{Finite_Quotients} which gives the bound $|G|\geq k! \cdot 2^{k-1}$ under the additional hypothesis that elements of $X$ pairwise commute. The totally symmetric set $X_n = \{(1 \; i) \mid i = 2,\ldots, n\}$ shows that the bound in \Cref{thm:main} is sharp. Our next theorem shows that $X_n$ is the only such example (with three exceptions for small $n$).

\begin{theorem}
\label{thm:classification}
Let $Y = \{y_1,\ldots,y_k\} $ be a totally symmetric set in $S_n$ of cardinality $k$.

\begin{enumerate}
    \item If $n \not \in \{3,4,6\}$ and $k = n-1$, then $Y$ is conjugate to $X_n$.
    \item If $n = 6$ and $k = 5$, then $Y$ is conjugate to either $X_6$ or $\rho(X_6)$ where $\rho \in \Out(S_6)$ is non-trivial.
    \item If $n = 4$ and $k = 3$, then $Y$ is either conjugate to $X_4$, conjugate to $\{(1\:2), (1\:3), (2\:3)\}$ or equal to $\{(1 \: 2)(3\:4),(1\:3)(2\:4),(1\:4)(2\:3)\}$.
    \item If $n = 3$, then $Y$ may be any subset of any conjugacy class of $S_3$. In particular $k \leq 3$, and equality is realized by $\{(1\:2), (1\:3), (2\:3)\}$.
\end{enumerate}
\end{theorem}

Both the braid group $B_n$ and the general linear group $\GL_n(\C)$ have similarly rigid maximal totally symmetric sets (see \cite[Lemma 2.6]{Kordek-Margalit} and \cite[Theorem B]{Caplinger-Salter} respectively). This is not a completely general phenomenon---$\Z_2 \times S_n$ contains two non-conjugate maximal totally symmetric sets---but it raises the question of what general properties lead to such rigidity. This question can also be asked about ``totally symmetric" objects not in groups (see \Cref{actionDef}). The totally symmetric multicurves of \cite{Kordek-Margalit} and the totally symmetric arrangements of \cite{Caplinger-Salter} both exhibit similar rigidity properties, and are both crucial in the proofs of their corresponding rigidity theorems.

As a sample application of \Cref{thm:classification}, we use this result to give a short, conceptually simple proof of the well-known classification of homomorphisms $S_n \to S_m$ for $n \geq m$ due to Hölder \cite{Holder}. The basic idea is that \Cref{thm:classification} determines all possible images of $f(X_n)$. 

\begin{theorem}[Hölder]
\label{corr}
Let $n \geq m > 2$ and $f:S_n \to S_m$ be a homomorphism. Then

\begin{enumerate}
     \item \label{smaller_then_cyclic}If $n > m$ and $(n,m) \neq (4,3)$, then $\Image(f)$ is cyclic. 
    \item \label{equal_then_inner}If $n = m \not \in \{4,6\}$ and $\Image(f)$ is non-cyclic, then $f$ is an inner automorphism.
    \item \label{out} If $n = m = 6$ and $\Image(f)$ is non-cyclic, then $f$ is an automorphism. Furthermore, $\Out(S_6) \cong \Z /2\Z$.
    \item \label{exceptional} If $(n,m) = (4,3)$ and $\Image(f)$ is non-cyclic, then $f$ is conjugate to the exceptional map $g: S_4 \to S_3$ defined by $g(1 \: 4) = (1 \: 2)$, $g(2 \: 4) = (1 \: 3)$ and $g(3 \: 4) = (2 \: 3)$.
    \item \label{exceptional4} If $n = m = 4$ and $\Image(f)$ is non-cyclic, then $f$ is either an inner automorphism or conjugate to the exceptional map above composed with the inclusion $S_3 \to S_4$.
\end{enumerate}

\end{theorem}

\subsection{Acknowledgments} 

The author would like to thank Dan Margalit and Dan Minahan for their suggestion to push \Cref{thm:main} farther than a simple bound. He is also grateful to Dan Margalit for his helpful comments.

\section{Totally Symmetric Sets}

Informally, a totally symmetric set is a subset $Y \subset G$ of a group such that every permutation of $Y$ can be realized by conjugation in $G$---the conjugation action contains every symmetry of $Y$. 

\begin{definition}[Totally symmetric set of a group]
Let $G$ be a group. A subset $Y = \{y_1,\ldots, y_k\} \subset G$ is said to be totally symmetric if for every $\sigma \in S_k$, there is some $g_\sigma \in G$ such that $g_\sigma y_i g_\sigma\inv = y_{\sigma(i)}.$
\end{definition}

The definition of total symmetry given by Kordek-Margalit \cite{Kordek-Margalit} made the additional constraint that the elements pairwise commute. In \cite{Caplinger-Salter}, Salter and the author generalized this definition to arbitrary $G$-sets and in particular required the use of non-commuting totally symmetric sets. We do not require the full apparatus of the definition from \cite{Caplinger-Salter}, but we will make use of this broader notion of total symmetry.

\begin{definition}[General totally symmetric set]
\label{actionDef}
Let $Z$ be a $G$-set. A subset $Y = \{y_1,\ldots, y_k\} \subset Z$ is said to be totally symmetric if for every $\sigma \in S_k$, there is a $g_\sigma \in G$ such that $g_\sigma \cdot y_i = y_{\sigma(i)}$.
\end{definition}

This group action perspective will be quite useful in our analysis. The stabilizer of a totally symmetric set (under the induced action on subsets) will play a central role in the proofs of Theorems \ref{thm:main} and \ref{thm:classification}. In fact, \Cref{actionDef} can be reformulated in terms of the stabilizer: a subset $Y \subset Z$ is totally symmetric if and only if the natural map $\Stab(Y) \to S_Y$ is surjective.

The utility of totally symmetric sets stems from ``fundamental lemma" of \cite{Kordek-Margalit}, which we now state in the language of \Cref{actionDef}.

\begin{lemma}[Collision implies collapse]
\label{colImpCol}
Let $G$ be a group, $Z_1,Z_2$ be $G$-sets, and $Y = \{y_1,\ldots,y_k\}\subset Z_1$ be a totally symmetric set. Let $f:Z_1 \to Z_2$ be a $G$-equivariant map. Then either $|f(Y)| = |Y|$ or $|f(Y)| = 1$. Furthermore, $f(Y)$ is totally symmetric.
\end{lemma}

\begin{proof}
For $|Y| \leq 2$, the result is clear. Say $|Y| > 2$ and $f(y_i) = f(y_j)$ for distinct $y_i,y_j \in Y$. For every $y_m \in Y$ distinct from $y_i,y_j$, there is some $g_{(j \: m)} \in G$ realizing the transposition $(j \: m)$ on $Y$. Then 

$$f(y_i) = f(g_{(j \: m)}\cdot y_i) = g_{(j \: m)}\cdot f(y_i)= g_{(j \: m)} \cdot f( y_j) = f(g_{(j \: m)}\cdot y_j) = f(y_{m}).$$

Any singleton is vacuously totally symmetric. If $|f(Y)| = |Y|$, then $g_\sigma\cdot f(y_i) = f(y_{\sigma(i)})$, so $f(Y)$ is also totally symmetric. 
\end{proof}

When $Z_1,Z_2$ are groups under the action of conjugation and $f$ is a homomorphism, we recover \cite[Lemma 2.1]{Kordek-Margalit}, which states that if $f:G \to H$ is a homomorphism, and $X\subset G$ is totally symmetric, then $f(X)$ is also totally symmetric and has cardinality $1$ or $|f(X)|$. This is the primary way totally symmetric sets are used to study homomorphisms. As a sample application, we will prove \Cref{corr} parts \ref{smaller_then_cyclic} and \ref{equal_then_inner} assuming Theorem \ref{thm:classification}. Conceptually, this proof is quite simple---in order to find all maps $f:S_n \to S_m$, we need only find the possible images of $X_n$, which are listed in \Cref{thm:classification}. In part  \ref{smaller_then_cyclic}, there are no suitable images, and in part \ref{equal_then_inner} there is only one.

\begin{proof}[Proof \Cref{corr}]

Let $n \geq m > 2$ be integers and $f:S_n \to S_m$ be a homomorphism. We prove only parts \ref{smaller_then_cyclic} (if $n > m$, then $\Image(f)$ is cyclic) and \ref{equal_then_inner} (if $n = m \not \in \{4,6\}$ and $\Image(f)$ is non-cyclic, then $f$ is an inner automorphism). Parts \ref{out}, \ref{exceptional} and \ref{exceptional4} are similar. 

Consider the totally symmetric set $f(X_n)$, which has cardinality $1$ or $n-1$ by \Cref{colImpCol}.  If $n > m$ with $(n,m) \neq (4,3)$, Theorem \ref{thm:classification} says that $S_m$ has no totally symmetric sets of size $n-1$. Then $|f(X_n)| = 1$, which implies that $\Image(f)$ is cyclic. We will proceed assuming $\Image(f)$ is non-cyclic. Then $f(X_n)$ has cardinality $n-1$ and therefore must be one of the totally symmetric sets listed in \Cref{thm:classification}.

If $n \not \in \{3,4,6\}$ and $n = m$, then \Cref{thm:classification} gives some $\sigma \in S_n$ so that $f(X_n) = \tilde{\sigma} (X_n)$, where $\tilde{\sigma}$ is the inner automorphism corresponding to $\sigma$. Then $(\tilde{\sigma}\inv \circ f)(X_n) = X_n$, that is $\tilde{\sigma}\inv \circ f$ permutes $X_n$. Total symmetry now gives an element $\tau \in S_n$ realizing this permutation so that $\tilde{\tau}\inv \circ \tilde{\sigma}\inv \circ f$ is the identity map on $X_n$. Then $f = \widetilde{\sigma \tau}$ is an inner automorphism as desired.
\end{proof}

\section{Proof of \Cref{thm:main}}

Let $X$ be a totally symmetric set of cardinality $k$ in a group $G$. The proof of \Cref{thm:main} is essentially an orbit-stabilizer argument applied to the action of conjugation on totally symmetric sets.

\begin{proof}[Proof of \Cref{thm:main}] By total symmetry, the natural map $\phi:\Stab(X) \to S_k$ is surjective, and therefore $|\Stab(X)| \geq k!$. It remains to show $|\Orb(X)| \geq k+1$. This is accomplished in two stages: first, we argue that some $Y \in \Orb(X)$ intersects $X$ non-trivially, then we use this intersection to produce $k$ additional totally symmetric sets in $\Orb(X)$.

Assume every $Y \in \Orb(X)$ is disjoint from $X$. In particular, any $a \in X$ satisfies $aXa\inv = X$, meaning $X \subset \Stab(X)$ is a totally symmetric set in $\Stab(X)$. Moreover, $X$ is an entire conjugacy class of $\Stab(X)$. If any $a \in X$ satisfies $\phi(a) = e$, then the elements of $X$ pairwise commute, and we may apply \cite[Proposition 2.2]{Finite_Quotients} which gives the bound $|G| \geq k!\cdot 2^{k-1}>(k+1)!$ for any totally symmetric set with pairwise commuting elements. Then we may additionally assume $\phi(X) \neq \{e\}$. Because $\phi:\Stab(X) \to S_k$ is surjective, $\phi(X)$ is also an entire conjugacy class of $S_k$ consisting of elements fixing at least one point. For $k > 3$, such conjugacy classes in $S_k$ have cardinality larger than $k$. This contradicts $|X| = k$.

Let $Y \in \Orb(X)$ intersect $X$ non-trivially. we can use total symmetry to produce $\binom{k}{|X \cap Y|} \geq k$ other elements of $\Orb(X)$ as follows. For each $|X\cap Y|$-element subset $A \subset X$, let $g_A \in \Stab(X)$ be such that $g_A(X \cap Y)g_A\inv = A$. Then the $g_AYg_A\inv$ are all distinct, as they have different intersections with $X$. This proves $|G|\geq (k+1)!$.

\vspace{.3cm}

We will now prove the second part of \Cref{thm:main}, which states that $S_{k+1}$ is the only group $G$ of order $(k+1)!$ with a totally symmetric set $X$ of cardinality $k > 3$. The basic strategy is to produce an action of $G$ on a $(k+1)$-element set which is isomorphic to the action of $S_{k+1}$.

From the proof of part 1, the equality $|G| = (k+1)!$ is achieved exactly when $\phi:\Stab(X) \to S_k$ is an isomorphism and $|\Orb(X)| = k+1$. In this case, every $Y \in \Orb(X)$ with $Y \neq X$ satisfies $|X \cap Y| = 1$ or $|X \cap Y| = k+1$. We deal with the case $|X \cap Y| = 1$; the other case is nearly identical. 

Write $X = \{x_1,\ldots, x_k\}$, let $f:G \to S_{\Orb(X)} \cong S_{k+1}$ be the action of $G$ on $\Orb(X)$ and let $Y_i \in \Orb(X)$ denote the unique totally symmetric set satisfying $X \cap Y_i = \{x_i\}$. Then $\Orb(X) = \{X, Y_1,\ldots Y_k\}$. Furthermore, $S_k \cong \Stab(X) \subset G$ acts on $\{Y_1,\ldots, Y_k\}$ by permuting indices. Any $g \not \in \Stab(X)$ does not fix $X$, so $f(\Stab(X))$ and $f(g)$ generate $S_{k+1}$. \end{proof}

\section{Proof of \Cref{thm:classification}}

Let $Y = \{y_1,\ldots, y_k\} \subset S_{n}$ be a totally symmetric set. In this section, we will prove that if $k$ takes the largest value allowed by \Cref{thm:main}, then it must be one of the totally symmetric sets listed in \Cref{thm:classification}. The case $n=3$ is dealt with by noting that any subset of any conjugacy class of $S_3$ is totally symmetric. The cases $n = 4$ and $n \geq 5$ will be treated separately. The proof requires the following two facts about permutation groups, both of which can be found in \cite[Chapter 5, Section 2]{dixon2012permutation}. 

\begin{enumerate}
    \item Any proper subgroup $H \subset S_n$ not equal to $A_n$ satisfies $|S_n: H| \geq n$.
    \item For $n \neq 6$, every index $n$ subgroup of $S_n$ is a point stabilizer (that is, a subgroup $S_{n-1} \subset S_n$ fixing a point in $[n] = \{1,\ldots, n\}$). If $n = 6$, there is one additional conjugacy class of point stabilizers found by applying an outer automorphism to a point stabilizer.
\end{enumerate}

\subsection{The generic case: \boldmath{$n \geq 5$}.}

We first claim that $\Stab(Y)$ is a point stabilizer or else $n=6$ and $\Stab(Y)$ is the image of a point stabilizer under an outer automorphism of $S_6$. Since $Y$ is totally symmetric, the natural map $\Stab(Y) \to S_k \cong S_{n-1}$ is surjective. Counting orders shows that $|S_n:\Stab(Y)| \leq n$. Then $\Stab(Y)$ is a proper subgroup of $S_n$ not equal to $A_n$, meaning $|S_n : \Stab(Y)| \geq n$ and therefore $|S_n : \Stab(Y)| = n$. We now invoke the fact (see \cite[Chapter 5, Section 2]{dixon2012permutation}) that all index $n$ subgroups of $S_n$ are point stabilizers when $n \neq 6$ and when $n=6$, there is an additional conjugacy class of $n$ subgroup. This proves the claim.

We now show that if $\Stab(Y)$ is a point stabilizer, then $Y$ is conjugate to $X_n = \{(1 \; i) \mid i = 2,\ldots,n \}$. For future notational simplicity, we will actually show $Y$ is conjugate to $\{(i \; n) \mid i = 1,\ldots,n-1 \}$. This suffices even when $n=6$---if $\Stab(Y)$ is not a point stabilizer, then let $[\rho] \in \Out(S_6)$ and consider the totally symmetric set $\rho(Y)$ and its stabilizer $\Stab(\rho(Y)) = \rho(\Stab(Y))$.

Without loss of generality, assume $\Stab(Y)$ fixes the point $n \in [n]$. Set $a_i = y_i(n)$, and for $\sigma \in S_{Y}$, let $g_\sigma \in \Stab(Y)$ be such that $g_\sigma y_i g_\sigma\inv = y_{\sigma (i)}$. Then
\begin{equation}
\label{TSnumbers}
g_\sigma (a_i) = g_\sigma y_i(n) = g_\sigma y_i g_\sigma\inv(n)  = y_{\sigma (i)}(n) = a_{\sigma (i)}. 
\end{equation}

In other words, the set $\{a_i \mid i \in [k]\} \subset [n]$ is totally symmetric in the sense of \Cref{actionDef}. The association $y_i \to a_i$ is moreover $\Stab(Y)$-equivariant.

We next claim that $\{a_i \mid i \in [k]\} = \{1,\ldots, n-1\}$. If any $a_i$ is equal to $n$, then every $a_i$ is equal to $n$ by \Cref{TSnumbers} (recall every $g_\sigma \in \Stab(Y)$ fixes $n$). If $a_i = a_j$ for distinct $i,j \in [k]$, then \Cref{colImpCol} tells us $\{a_m \mid m \in [k]\}$ is the singleton $ \{a_i\}$. \Cref{TSnumbers} now says, $\Stab(Y)$ fixes $a_i \neq n$ in addition to $n$. Then $\Stab(Y)$ has order at most $(n-2)!$ and therefore cannot surject onto $S_k \cong S_{n-1}$. Thus, $\{a_i \mid i \in [k]\}$ is an $(n-1)$-element subset of $[n]$ not containing $n$, so it is $\{1,\ldots, n-1\}$ as claimed.

By conjugating, we may assume without loss of generality that $a_i = i$. \Cref{TSnumbers} now says $g_\sigma = \sigma$. At this stage, we know $y_i(n) = i$ and want to show that $y_i = (i \; n)$. This will be accomplished in two steps:

\begin{enumerate}
    \item Show that $y_i(i) = n$
    \item Show that if $j \not \in \{i,n\}$, then $y_i(j) = j$
\end{enumerate}

\vspace{.1cm}
\noindent
\textit{Step 1.} Suppose $y_i(i)  = j\neq n$. Let $k \not \in \{n,i,j\}$, and consider the element $g_{(j \: k)} = (j \: k)$. By total symmetry, $$(j \: k) y_i (j \: k) = g_{(j \: k)}y_ig_{(j \: k)}\inv = y_i.$$ But the left hand side of this equation sends $i \to k$, while the right hand side sends $i \to j$. Hence, $y_i(i) = n$. 

\vspace{.1cm}
\noindent
\textit{Step 2.} Assume that $y_i$ does not fix some $j \not \in \{i,n\}$, that is $y_i(j) = k$ for $j,k \not \in \{i,n\}$. We will use the same trick: Because $n \geq 5$, there is an $m \not \in \{i,n,j,k\}$. Just as before, $$(m \: j) y_i (m \: j) = g_{(m \: j)}y_ig_{(m \: j)}\inv = y_i.$$ But the left hand side takes $m \to k$, while the right hand side takes $j \to k$. Then $y_i = (i \: n)$ as required. Note that this step fails for the totally symmetric set $\{(1 \: 2)(3\:4),(1\:3)(2\:4),(1\:4)(2\:3)\} \subset S_4$.

\subsection{The exceptional case: \boldmath{$n = 4$}}

We will check $3$ element subsets of conjugacy classes directly. Elements of a totally symmetric set are conjugate, so there are four cycle types to consider:

\begin{enumerate}
    \item $(\ast \: \ast)(\ast \: \ast)$
    \item $(\ast \: \ast)$
    \item $(\ast \: \ast \: \ast)$
    \item $(\ast \: \ast \: \ast \: \ast)$
\end{enumerate}
\vspace{.3cm}

\noindent\textit{Case 1: $(\ast \: \ast)(\ast \: \ast)$.} We claim the conjugacy class of cycle type $(\ast \: \ast)(\ast \: \ast)$ given by $Y = \{y_1 = (1 \: 2)(3\: 4),y_2 = (1\:3)(2\:4),y_3 = (1\:4)(2\:3)\}$ is totally symmetric. Indeed, $g_{(1 \; 2)} = (2 \: 3)$, $g_{(1\: 3)} = (1 \: 3)$ and $g_{(2\: 3)} = (2\; 3)$ realize all three transpositions of $Y$.

\vspace{.3cm}

We are then left to consider totally symmetric sets $Y = \{c_1,c_2,c_3\}$ of cycles. By total symmetry, the intersection pattern of subsets of $Y$ must be the same: $|c_1\cap c_2| = |c_2\cap c_3| = |c_1\cap c_3|$ (here we think of cycles as subsets of $[4] = \{1,2,3,4\}$). This is a consequence of the more general fact\footnote{In the language of \cite{Caplinger-Salter}, the set of intersections $c_i\cap c_j$ form a totally symmetric set under the action of $S_4$ on two elements subsets of $[4]$. This is a more general notion of total symmetry than used in this paper.} that $g_\sigma (c_i \cap c_j) = g_\sigma c_i \cap g_\sigma c_j = c_{\sigma(i)} \cap c_{\sigma(j)}$, where again we think of the cycles as subsets of $[4]$.

\vspace{.3cm} 

\noindent\textit{Case 2: $(\ast \: \ast)$.} There are no three element sets of transpositions in $S_4$ which do not intersect. Then $c_1$ and $c_2$ share a single element. By the previous discussion, $c_3$ must intersect both $c_1$ and $c_2$. Then either $c_3$ contains the point $c_1 \cap c_2$ or it contains the points $c_1 \setminus (c_1\cap c_2)$ and $c_2 \setminus (c_1\cap c_2)$. In the first case, $Y$ is conjugate to $X_4$. In the second, $Y$ is conjugate to $\{(1\:2), (1\:3), (2\:3)\}$.

\vspace{.3cm}

\noindent\textit{Case 3: $(\ast \: \ast \: \ast)$.} If $c_1 = c_2$ as subsets of $[4]$, then $c_3$ is a $3$-cycle on the same three elements. But there are only two distinct 3-cycles in $S_3$. Then $c_1 \neq c_3$ as subsets of $[4]$, and $|c_1\cap c_2| = 2$. Then $c_3$ intersects $c_1\cap c_2$ in exactly one point---if it contained $c_1\cap c_2$, then the third point (which lies in either $c_1$ or $c_2$) would break the symmetry of intersection patterns. Let $p$ be the unique point in $c_1 \cap c_2 \cap c_3$, and let $g_{(1\:2)}$ realize the permutation $(1\: 2)$ on the totally symmetric set $Y = \{c_1,c_2,c_3\}$. The centralizer of a 3-cycle in $S_4$ is the group generated by that 3-cycle. Then $g_{(1\: 2)}$ is a power of $c_3$, but also fixes $p$. Then $g_{(1\: 2)} = e$, which does not realize $(1\: 2)$.

\vspace{.3cm}

\noindent \textit{Case 4: $(\ast \: \ast \: \ast \: \ast)$.} Any group element $c$ and its inverse cannot appear in a three element totally symmetric set---one cannot move $c$ by conjugation while fixing $c\inv$. There are only three inverse pairs of four-cycles in $S_4$, so $c_1 = (1 \: 2 \: 3 \: 4) ^{\pm 1}$, $c_2 = (1 \: 2 \: 4 \: 3)^{\pm 1}$, and $c_3 = (1 \: 3 \: 2 \: 4)^{\pm 1}$. Without loss of generality, assume $c_1 = (1 \: 2 \: 3 \: 4)$. We will show $c_2 = (1 \: 2 \: 4 \: 3)$ is impossible---the case $c_2 = (1 \: 2 \: 4 \: 3)^{-1 }$ is nearly identical. 

As before, the element $g_{(2 \: 3)}$ realizing the permutation $(2 \: 3)$ must be a power of $c_1$. Denote this power $p \in \{0,1,2,3\}$. If $p = 3$, then $(c_1^3)^3 = c_1$ also realizes the permutation $(2 \: 3)$, so we need only check $p = 1$ and $p = 2$. If $p = 1$, we compute $$c_3 = c_1 c_2 c_1\inv = (2 \: 3 \: 1 \: 4) \qquad \text{and} \qquad c_2 = c_1 c_3 c_1\inv = (3 \: 4 \: 2 \: 1) \neq c_2.$$ Then $p = 2$, and $$c_3 = c_1^2 c_2 c_1^{-2} = (3 \: 4 \: 2 \: 1) = c_2\inv.$$ But $Y$ cannot contain both $c_2$ and $c_2\inv$.

\bibliographystyle{alpha}
\bibliography{bib}

\end{document}